\documentclass[12pt]{article}

\usepackage{amssymb}

\newtheorem{theorem}{Theorem}[section]
\newtheorem{lemma}[theorem]{Lemma}
\newtheorem{corollary}[theorem]{Corollary}
\newtheorem{assumption}[theorem]{Assumption}

\newenvironment{proof}{\noindent {\bf
Proof}.\ }{\proofbox\par\smallskip\par}

\def\numberlikeadb{\global\def\theequation{\thesection.\arabic{equation}}}
\numberlikeadb

\newcommand{\halmos}{\rule{1ex}{1.4ex}}
\newcommand{\proofbox}{\hspace*{\fill}\mbox{$\halmos$}}

\newcommand{\E}{\mbox{\bf E}}
\newcommand{\pr}{\mbox{\bf P}}

\newcommand{\reals}{{R}}
\newcommand{\ints}{\mathbb{Z}}

\newcommand{\eqa}{\begin{eqnarray}}
\newcommand{\ena}{\end{eqnarray}}
\newcommand{\eq}{\begin{equation}}
\newcommand{\en}{\end{equation}}
\newcommand{\eqs}{\begin{eqnarray*}}
\newcommand{\ens}{\end{eqnarray*}}

\def\l{\lambda}

\def\a{\alpha}
\def\b{\beta}
\def\g{\gamma}

\def\d{\delta}
\def\D{\Delta}

\def\h{\eta}
\def\z{\zeta}
\def\th{\theta}

\def\k{\kappa}
\def\m{\mu}
\def\n{\nu}

\def\r{\rho}
\def\s{\sigma}

\def\t{\tau}
\def\f{\varphi}

\def\nin{\noindent}

\def\msk{\medskip}

\def\Blb{\left\{}
\def\Brb{\right\}}

\def\giv{\,|\,}

\def\non{\nonumber}
\def\integ{{\mathbb Z}}

\def\Eq{\ =\ }
\def\DEq{\ :=\ }
\def\Le{\ \le\ }

\def\sjo{\sum_{j\ge0}}
\def\sji{\sum_{j\ge1}}
\def\sio{\sum_{i\ge0}}
\def\sii{\sum_{i\ge1}}
\def\slo{\sum_{l\ge0}}

\def\Ref#1{{\rm (\ref{#1})}}
\def\vol#1{{\bf #1},}
\def\bp{\begin{proof}}
\def\ep{\end{proof}}

\def\bone{{\bf 1}}

\def\Bl{\left(}
\def\Br{\right)}

\def\half{{\textstyle{\frac12}}}

\def\ff{{\cal F}}
\def\xx{{\cal X}}

\def\ignore#1{}
\def\ex{\E}

\def\sjni{\sum_{j\neq i}}

\def\tmax{t_{\max}}

\def\Z{\ints}
\def\Zp{\Z_+}

\def\ssjt{\sum_{j\colon \s_j \le t}}

\def\R{\reals}
\def\rr{{\mathcal R}}
\def\jj{{\mathcal J}}
\def\sjj{\sum_{J\in \jj}}
\def\x{\xi}
\def\smo{\sum_{m\ge0}}
\def\sko{\sum_{k\ge0}}
\def\sro{\sum_{r\ge0}}
\def\nm#1{{\|#1\|_\m}}
\def\uii{^{(i)}}

\def\uN{^{(N)}}
\def\Def{\DEq}
\def\tM{{\widetilde M}}
\def\tm{{\widetilde m}}

\def\NNN{\tM_N}
\def\ZZZ{Z}
\def\nnn{\tm_N}
\def\Ge{\ \ge\ }

\def\gga{\gamma}
\def\rmax{r_{\rm max}}
\def\sutt{\sup_{0\le t\le T}}
\def\sflnn{\sqrt{\frac{\log N}N}}
\def\ui{^{(1)}}
\def\uii{^{(i)}}
\def\ut{^{(2)}}
\def\txn{t^X_n}
\def\txnn{t^{X_N}_n}
\def\tilt{{\tilde\t}}
\def\halfh{{\textstyle{\frac14}}}
\def\txi{t^X_\infty}
\def\txin{t^{X_N}_\infty}
\def\wtxi{\wedge\txi}
\def\wtxn{\wedge\txn}
\def\wtxin{\wedge\txin}
\def\wtxnn{\wedge\txnn}
\def\lnti{\lim_{n\to\infty}}
\def\wtkb{\wedge\tkb}
\def\tkb{\tilt_k\uN(B)}

\def\wtaz{\wedge\t\uN(a,\z)}
\def\suttwt{\sup_{0\le t\le T\wtaz}}
\def\tC{{\widetilde C}}
\def\bbt{b}
\def\tr{\tilde r}
\def\cof{\partial}
\def\Blm{\left|}
\def\Brm{\right|}
\def\aaa{{\mathbb A}}
\def\tiNC{\t_1\uN(C)}
\def\tiNCd{\t_1\uN(C')}

\begin{document}

\title{A law of large numbers approximation for Markov population
processes with countably many types}
\author{
A. D. Barbour\footnote{Angewandte Mathematik, Universit\"at Z\"urich,
Winterthurertrasse 190, CH-8057 Z\"URICH;
ADB was supported in part by Schweizerischer Nationalfonds Projekt Nr.\
20--107935/1.\msk}
\ and
M. J. Luczak\footnote{London School of Economics;
MJL was supported in part by a STICERD grant.
}\\
Universit\"at Z\"urich and London School of Economics
}

\date{}
\maketitle

\begin{abstract}
When modelling metapopulation dynamics, the influence of a single patch
on the metapopulation depends on the number of individuals in the patch.
Since the population size has no natural upper limit, this leads to
systems in which there are countably infinitely many possible types of
individual.  Analogous considerations apply in the transmission of
parasitic diseases.  In this paper, we prove a law of large numbers for
quite general systems of this kind, together with a rather sharp
bound on the rate
of convergence in an appropriately chosen weighted $\ell_1$ norm.
\end{abstract}

 \noindent
{\it Keywords:} Epidemic models, metapopulation processes, countably many types,
  quantitative law of large numbers, Markov population processes \\
{\it AMS subject classification:} 92D30, 60J27, 60B12 \\
{\it Running head:}  A law of large numbers approximation

\section{Introduction}\label{introduction}
\setcounter{equation}{0}
There are many biological systems that consist of entities that
differ in their influence according to the number of active elements
associated with them, and can be divided into types accordingly.  
In parasitic diseases (Barbour \& Kafetzaki~1993,
Luchsinger~2001a,b, Kretzschmar~1993), the infectivity of a host depends on
the number of parasites that it carries; in metapopulations, the migration
pressure exerted by a patch is related to the number of its inhabitants
(Arrigoni~2003); the behaviour of a cell may depend on the number of copies
of a particular gene that it contains (Kimmel \& Axelrod~2002, Chapter~7); 
and so on.
In none of these examples is there a natural upper limit to the number of
associated elements, so that the natural setting for a mathematical model
is one in which there are countably infinitely many possible types of
individual.  In addition, transition rates typically increase with the
number of associated elements in the system --- for instance, each
parasite has an individual death rate, so that the overall death rate
of parasites grows at least as fast as the number of parasites --- and
this leads to processes with unbounded transition rates.  This paper is
concerned with approximations to density dependent
Markov models of this kind, when the typical population size~$N$ becomes
large.

In density dependent Markov population processes with only finitely many
types of individual, a law of large numbers approximation, in the form of
a system of ordinary differential equations, was established by Kurtz~(1970),
together with a diffusion approximation (Kurtz, 1971).  In the infinite
dimensional case, the law of large numbers was proved for some specific
models (Barbour \& Kafetzaki~1993, Luchsinger~2001b, Arrigoni~2003, 
see also L\'eonard~1990), using individually tailored methods.  A more general 
result was then given by Eibeck \& Wagner (2003).
In Barbour \& Luczak~(2008), the law of large numbers was strengthened by the addition of 
an error bound in~$\ell_1$ that is close to optimal order in~$N$.
Their argument makes use of an intermediate approximation involving an
independent particles process, for which the law of large numbers is
relatively easy to analyse.  This process is then shown to be sufficiently
close to the interacting process of actual interest, by means of a coupling
argument.  However, the generality of the results obtained is limited by
the simple structure of the intermediate process, and the model of
Arrigoni~(2003), for instance, lies outside their scope.

In this paper, we develop an entirely different approach, which circumvents
the need for an intermediate approximation, enabling a much wider class of
models to be addressed. The setting is that of families of Markov population 
processes $X_N := (X_N(t),\,t\ge0)$, $N\ge1$, taking values in the
countable space $\xx_+ := \{X \in \Zp^{\Zp};\, \smo X^m < \infty\}$.
Each component represents the number of individuals of a particular type,
and there are countably many types possible; however, at any given time,
there are only finitely many individuals in the system.  The process evolves
as a Markov process with state-dependent transitions
\eq\label{1.0}
    X \ \to\ X + J \quad\mbox{at rate}\quad N\a_J(N^{-1}X),\qquad X \in \xx_+,\ J \in \jj,
\en
where each jump is of bounded influence, in the sense that
\eq\label{finite-jumps}
   \jj \subset \{X \in \Z^{\Zp};\, \smo |X^m| \le J_* < \infty\},
     \quad\mbox{for some fixed}\quad J_* < \infty, 
\en
so that the number of individuals affected is uniformly bounded.  Density
dependence is reflected in the fact that the arguments of the functions~$\a_J$
are counts normalised by the `typical size'~$N$.
Writing $\rr := \R_+^{\Zp}$, the
functions~$\a_J\colon\, \rr  \to \R_+$ are assumed to satisfy
\eq\label{finite-rates}
    \sjj \a_J(\x) \ < \ \infty,\qquad \x \in \rr_0,
\en
where $\rr_0 := \{\x \in \rr\colon\, \x_i = 0 {\mbox{ for all but finitely many }} i\}$;
this assumption implies that the processes~$X_N$ are pure jump processes, at least
for some non-zero length of time.  To prevent
the paths leaving~$\xx_+$, we also assume that $J_l \ge -1$ for each~$l$, and
that $\a_J(\x) = 0$ if $\x^l = 0$ for any $J\in\jj$ such that $J^l = -1$.
Some remarks on the consequences of allowing transitions~$J$ with $J^l \le -2$
for some~$l$ are made at the end of Section~\ref{approximation}.

The law of large numbers is then formally expressed in terms of the 
system of {\it deterministic equations\/}
\eq\label{determ}
   \frac{d\x}{dt} \Eq \sjj J\a_J(\x) \ =:\ F_0(\x),
\en
to be understood componentwise for those $\x\in\rr$ such that
\[
    \sjj |J^l| \a_J(\x) \ <\ \infty,\quad \mbox{for all } l \ge 0,
\]
thus by assumption including~$\rr_0$.  Here, the quantity~$F_0$ represents
the infinitesimal average drift of the components of the random process. 
However, in this generality,
it is not even immediately clear that equations~\Ref{determ} have a solution.

In order to make progress, it is assumed that the unbounded
components in the transition rates can be assimilated into a
linear part, in the sense that~$F_0$ can be written in the form
\eq\label{F-assn}
   F_0(\x) \Eq A\x + F(\x),
\en
again to be understood componentwise,  where~$A$ is a constant $\Zp\times\Zp$ matrix.
These equations are then treated as a perturbed linear system  (Pazy~1983, Chapter~6). 
Under suitable assumptions on $A$, there exists a measure~$\m$ on~$\Zp$, 
defining a weighted $\ell_1$ norm $\| \cdot \|_{\mu}$ on ${\mathcal R}$,  
and a strongly
$\| \cdot \|_{\mu}$--continuous semigroup $\{R(t),\,t\ge0\}$ of transition matrices 
having pointwise derivative $R'(0) = A$.  
If $F$ is locally $\| \cdot \|_{\mu}$--Lipschitz and 
$\nm{x(0)} < \infty$, this suggests using the solution~$x$ of the integral equation
\eq\label{mild-solution}
   x(t) \Eq R(t)x(0) + \int_0^t R(t-s)F(x(s))\,ds
\en
as an approximation to~$x_N := N^{-1}X_N$, instead of solving the deterministic 
equations~\Ref{determ} directly. We go on to
show that the solution~$X_N$ of the stochastic system can be expressed using  
a formula similar to~\Ref{mild-solution}, which has an additional stochastic component in the
perturbation: 
\eq\label{stochastic-solution}
   x_N(t) \Eq R(t)x_N(0) + \int_0^t R(t-s)F(x_N(s))\,ds + \tm_N(t),
\en
where
\eq\label{stochastic-perturbation}
   \tm_N(t) \Def \int_0^t R(t-s)\,dm_N(s),
\en
and~$m_N$ is the local martingale given by
\eq\label{local-mN}
  m_N(t) \Def x_N(t) - x_N(0) - \int_0^t F_0(x_N(s))\,ds.
\en
The quantity $m_N$ can be expected to be small, at least componentwise, under
reasonable conditions.

To obtain tight control over~$\tm_N$ in all components simultaneously,
sufficient to ensure that $\sup_{0\le s\le t}\nm{\tm_N(s)}$ is small, we 
derive Chernoff--like bounds on the deviations of the most significant
components, with the help of
a family of exponential martingales.  The remaining components are
treated using some general {\it a priori\/} bounds on the behaviour of
the stochastic system.  This allows us to take the difference between the
stochastic and deterministic equations \Ref{stochastic-solution} 
and~\Ref{mild-solution}, after which a Gronwall argument can be carried
through, leading to the desired approximation.

The main result, Theorem~\ref{main-theorem}, guarantees an approximation error of
order $O(N^{-1/2}\sqrt{\log N})$ in the weighted $\ell_1$ metric $\nm{\cdot}$, 
except on an
event of probability of order $O(N^{-1}\log N)$.  More precisely, for each $T > 0$, 
there exist constants $K_T^{(1)}, K_T^{(2)}, K_T^{(3)}$ such 
that, for $N$ large enough, if  
$$
  \|N^{-1} X_N(0) - x(0)\|_{\mu} \le K_T^{(1)} \sqrt{\frac{\log N}{N}},
$$
then
\begin{equation}
\label{main-result-1}
    \pr \Big  ( \sup_{0 \le t \le T} \|N^{-1} X_N (t) - x(t) \|_{\mu} 
      > K_T^{(2)} \sqrt{\frac{\log N}{N}}   \Big ) \le K^{(3)}_T \frac{\log N}{N}.
\end{equation}
The error bound is sharper, by a factor of $\log N$, than that given in
Barbour \& Luczak~(2008), and the theorem is applicable to a much wider
class of models.  However, the method of proof involves moment arguments,
which require somewhat stronger assumptions on the initial state of the
system, and, in models such as that of Barbour \& Kafetzaki~(1993), on the choice
of infection distributions allowed. The conditions under which
the theorem holds can be divided into three categories: growth conditions
on the transition rates, so that the {\it a priori\/} bounds, which have
the character of moment bounds, can be established; conditions on the matrix~$A$,
sufficient to limit the growth of the semigroup~$R$, and (together with the 
properties of~$F$) to determine the weights
defining the metric in which the approximation is to be carried out; and conditions
on the initial state of the system. The {\it a priori\/} bounds are derived in
Section~\ref{a-priori}, the semigroup analysis is conducted in 
Section~\ref{semigroup-section}, and
the approximation proper is carried out in Section~\ref{approximation}.  The paper
concludes in Section~\ref{examples} with some examples.

The form~\Ref{stochastic-perturbation} of the stochastic component~$\tm_N(t)$
in~\Ref{stochastic-solution} is very similar to that of a key element in the
analysis of stochastic partial differential equations; see, for example,
Chow~(2007, Section~6.6).  The SPDE 
arguments used for its control are however typically conducted in a Hilbert 
space context. Our setting is quite different in nature, and it does not seem 
clear how to translate the SPDE methods into our context.

\section{A priori bounds}\label{a-priori}
\setcounter{equation}{0}
We begin by imposing further conditions on the transition rates of the
process~$X_N$, sufficient to constrain its paths to bounded subsets
of~$\xx_+$ during finite time intervals, and in particular to ensure 
that only finitely many jumps
can occur in finite time.  The conditions that follow have 
the flavour of moment
conditions on the jump distributions.  Since the index~$j\in\integ_+$
is symbolic in nature, we start by fixing
an $\n\in\rr$, such that~$\n(j)$ reflects in some sense the `size' of~$j$,
with most indices being `large': 
\eq\label{nu-limits}
  \n(j) \ge 1 \ \mbox{for all}\ j\ge0 \quad\mbox{and}\quad
      \lim_{j\to\infty}\n(j) = \infty.  
\en
We then define the analogues of higher empirical moments using the quantities 
$\n_r\in\rr$, defined by $\n_r(j) := \n(j)^r$, $r\ge0$, setting 
\eq\label{S-def}
   S_r(x) \Def \sjo \n_r(j)x^j \Eq x^T\n_r,\qquad x \in \rr_0,
\en
where, for $x \in \rr_0$ and $y \in \rr$, $x^Ty := \slo x_ly_l$. 
In particular, for $X \in \xx_+$, $S_0(X) = \|X\|_1$.
Note that, because of~\Ref{nu-limits}, for any $r\ge1$, 
\eq\label{Sr-limits}
\#\{X \in \xx_+\colon\,S_r(X)\le K\} < \infty \quad\mbox{for all}\quad K > 0.  
\en
To formulate the conditions that limit the growth of the empirical
moments of~$X_N(t)$ with~$t$, we also define
\eq\label{UV-defs}
    U_{r}(x) \Def \sjj \a_J(x) J^T \n_r; \quad
    V_{r}(x) \Def \sjj \a_J(x) (J^T \n_r)^2, \quad x\in\rr. \non 
\en
The assumptions that we shall need are then as follows.  

\begin{assumption}\label{ap-assns}
There exists a~$\nu$ satisfying~\Ref{nu-limits} and $\rmax\ui, \rmax\ut \ge 1$ 
such that, for all~$X\in\xx_+$, 
\eqa
     &&\sjj \a_J(N^{-1}X) |J^T \n_r| \ <\ \infty,\qquad 0\le r\le \rmax\ui,
       \label{3.5a} 
\ena
the case $r=0$ following from \Ref{finite-jumps} and~\Ref{finite-rates};
furthermore, for some non-negative constants $k_{rl}$, the inequalities
\eqa   
    U_{0}(x) &\le&  k_{01} S_0(x) + k_{04}, \non \\
    U_{1}(x) &\le&  k_{11} S_1(x) + k_{14},  \label{U-bnd} \\
    U_{r}(x) &\le&  \{k_{r1} + k_{r2} S_0(x)\} S_r(x) + k_{r4},
    \quad 2\le r\le \rmax\ui; \non
\ena
and
\eqa
    V_{0}(x) &\le& k_{03}S_1(x) + k_{05},\non\\
    V_{r}(x) &\le& k_{r3} S_{p(r)}(x) + k_{r5}, 
       \qquad 1\le r\le \rmax\ut,\label{V-bnd}
\ena
are satisfied,
where $1\le p(r) \le \rmax\ui$ for $1\le r\le \rmax\ut$.
\end{assumption}

\nin The quantities $\rmax\ui$ and~$\rmax\ut$ usually need to be reasonably 
large, if Assumption~\ref{zeta-assns} below is to be satisfied.

\medskip
Now, for~$X_N$ as in the introduction, we 
let $\txnn$ denote the time of its \hbox{$n$-th} jump, with $t^{X_N}_0 = 0$,
and set $\txin := \lnti \txnn$, possibly infinite. 
For $0 \le t < \txin$, we define
\eq\label{SUV-N-defs}
   S_r\uN(t) \Def S_r(X_N(t));\quad U_r\uN(t) \Def U_{r}(x_N(t));
     \quad V_r\uN(t) \Def V_{r}(x_N(t)), 
\en
once again with $x_N(t) := N^{-1}X_N(t)$, and also
\eq\label{tau-defs}
  \t_r\uN(C) \Def \inf\{t < \txin\colon\,S_r\uN(t) \ge NC\},\qquad r\ge0,
\en
where the infimum of the empty set is taken to be~$\infty$.
Our first result shows that $\txin = \infty$ a.s., and
limits the expectations of $S_0\uN(t)$
and~$S_1\uN(t)$ for any fixed~$t$.

In what follows, we shall write ${\mathcal F}_s^{(N)}= \sigma (X_N (u), 0 \le u \le s)$, 
so that $({\mathcal F}^{(N)}_s: s \ge 0)$ is the natural filtration of the process $X_N$.

\begin{lemma}\label{AP-0}
Under Assumptions~\ref{ap-assns}, $\txin = \infty$ a.s. Furthermore, for any $t\ge0$,
\eqs
   \ex\{S_0\uN(t)\} &\le& (S_0\uN(0) + N k_{04}t)  e^{k_{01}t};\\
   \ex\{S_1\uN(t)\} &\le& (S_1\uN(0) + N k_{14}t)  e^{k_{11}t}.
\ens
\end{lemma}

\bp
Introducing the formal generator $\aaa_N$ associated with~\Ref{1.0},
\eq\label{A-gen-def}
  \aaa_N f(X) \Def N\sjj \a_J(N^{-1}X)\{f(X+J) - f(X)\}, \qquad X\in\xx_+,
\en
we note that $NU_l(x) = \aaa_N S_l(Nx)$. Hence, if we define~$M_l\uN$ by
\eq\label{M0-def}
     M_l\uN(t) \Def S_l\uN(t) - S_l\uN(0)  - N\int_0^t U_l\uN(u) \,du,\qquad t\ge0,
\en
for $0\le l\le \rmax\ui$, it is immediate from \Ref{Sr-limits}, \Ref{3.5a} and~\Ref{U-bnd}
that the process $(M_l\uN(t\wedge\tiNC),\,t\ge0)$ is a zero mean $\ff\uN$--martingale 
for each~$C>0$.
%
\ignore{
Defining $N(s) := \inf\{m\colon\,t^{X_N}_m \ge s\}$, 
and using the martingale and Markov properties, it follows that
\eqs
   \ex\{M_l\uN(t^{X_N}_{N(s)}) I[N(s)\le n] \}
   &=& \ex\Blb \ex\{M_l\uN(t^{X_N}_{N(s)}) I[N(s)\le n] \giv \ff_s^{(N)} \} \Brb \\
   &=& \ex\{M_l\uN(s) I[N(s) \le n] \}.
\ens
Applying the optional stopping theorem to the index $(N(s)\wedge n)$, we have that
$$\ex \{M_l^{(N)} (t^{X_N}_{N(s) \wedge n}) \} = 0,$$
in other words,
$$\ex \{M_l^{(N)} (t^{X_N}_{N(s)}) I [N(s) \le n] \} +  \ex \{M_l^{(N)} (t^{X_N}_n) I [N(s) > n] \}= 0.$$
But $N(s) \le n$ if and only if $t^{X_N}_n \ge s$, so 
we deduce that 
\eq\label{stopped-zero}
   \ex\{M_l\uN(s\wedge\txnn)\} \Eq 0,
\en
for any $0\le s\le T$ and any $n\ge1$.
}
In particular, considering $M_1\uN(t\wedge\tiNC)$, it follows in view of~\Ref{U-bnd} that
\eqa
  \ex\{S_1\uN(t\wedge\tiNC)\}  &\le& S_1\uN(0)  
     + \ex\Blb \int_0^{t\wedge\tiNC} \{k_{11} S_1\uN(u) + Nk_{14}\} \,du \Brb \non\\
   &\le& S_1\uN(0)  
     +  \int_0^t (k_{11} \ex\{S_1\uN(u\wedge\tiNC)\} + Nk_{14}) \,du .\non
\ena
Using Gronwall's inequality, we deduce that
\eq\label{S1C-ex-bnd}
    \ex\{S_1\uN(t\wedge\tiNC)\} \Le (S_1\uN(0) + Nk_{14}t)e^{k_{11}t},
\en
uniformly in $C>0$, and hence that
\eq\label{S1C-prob-bnd}
  \pr\Bigl[\sup_{0\le s\le t} S_1(X_N(s)) \ge NC\Bigr] 
      \Le C^{-1}(S_1(x_N(0)) + k_{14}t)e^{k_{11}t}
\en
also.  Hence $\sup_{0\le s\le t} S_1(X_N(s)) < \infty$ a.s.\ for any~$t$, 
$\lim_{C\to\infty}\tiNC = \infty$ a.s., and,
from \Ref{Sr-limits} and~\Ref{finite-rates}, it thus follows that $\txin = \infty$ a.s.
The bound on $\ex\{S_1\uN(t)\}$ is now immediate, and that on $\ex\{S_0\uN(t)\}$
follows by applying the same Gronwall argument to $M_0\uN(t\wedge\tiNC)$.
\ignore{
For the bounds on the expectations, we use~\Ref{U-bnd}, together with the
martingales $(M_l\uN(t\wedge\txnn),\,t\ge0)$, $l=1,2$, to give
\eqa
  \ex\{S_l\uN(t)\bone_{[t,\infty)}(\txnn)\} &\le& \ex\{S_l\uN(t\wedge\txnn)\} 
      \label{bone-to-wedge} \\
}
\ignore{
  &\le& S_l\uN(0)  
     + \ex\Blb \int_0^{t\wedge\txnn} \{k_{l1} S_l\uN(u) + Nk_{l4}\} \,du \Brb \non\\
}
\ignore{
   &\le& S_l\uN(0)  
     +  \int_0^t (k_{l1} \ex\{S_l\uN(u\wedge\txnn)\} + Nk_{l4}) \,du ,\label{bone-2}
\ena
much as before. Gronwall's inequality now implies that
\eq\label{Esln-bnd}
    \ex\{S_l\uN(t\wedge\txnn)\} \Le (S_l\uN(0) + Nk_{l4}t)e^{k_{l1}t},\quad l=0,1,
\en
for any~$n$, and the  lemma follows from~\Ref{bone-to-wedge}, by monotone 
convergence applied to the variables $S_l\uN(t)\bone_{[t,\infty)}(\txnn)$.
}
\ep


\bigskip
The next lemma shows that, if any $T>0$ is fixed and~$C$ is chosen large enough, 
then,  with high probability,
$N^{-1}S_0\uN(t) \le C$ holds for all $0\le t\le T$.

\begin{lemma}\label{AP-1}
Assume that Assumptions~\ref{ap-assns} are satisfied, and that
$S_0\uN(0) \le NC_0$ and $S_1\uN(0) \le NC_1$.  Then, for any $C
\ge 2(C_0 + k_{04}T)e^{k_{01}T}$, we have
\[
      \pr\bigl[\{\t_0\uN(C)\le T\}\bigr] \Le (C_1\vee1)K_{00}/(NC^2),
\]
where $K_{00}$ depends on~$T$ and the parameters of the model.  
\end{lemma}

\bp
It is immediate from \Ref{M0-def} and~\Ref{U-bnd} that
\eqa
   S_0\uN(t) &=& S_0\uN(0)  + N\int_0^t U_0\uN(u) \,du + M_0\uN(t)
         \phantom{HHHHHHHHHHH}\non\\
     &\le&  S_0\uN(0) + \int_0^t (k_{01} S_0\uN(u) + Nk_{04})  \,du 
         + \sup_{0\le u\le t} M_0\uN(u).   \label{S-0-bnd}
\ena
Hence, from Gronwall's inequality, if $S_0\uN(0) \le NC_0$, then
\eq\label{S0-bnd}
   S_0\uN(t) \Le \left\{N(C_0 + k_{04}T) + \sup_{0\le u\le t}
        M_0\uN(u)\right\}e^{k_{01}t}.
\en

Now, considering the quadratic variation of~$M_0\uN$, we have
\eq\label{variance-MG}
   \ex\Blb \{M_0\uN(t\wedge\tiNCd)\}^2 - N\int_0^{t\wedge\tiNCd} V_0\uN(u)\,du \Brb
     \Eq 0
\en
for any $C' > 0$, from which it follows, much as above, that
\eqs
   \ex\Bl \{M_0\uN(t\wedge\tiNCd)\}^2 \Br &\le& \ex\Blb N\int_0^t V_0\uN(u\wedge\tiNCd)\,du \Brb\\
      &\le& \int_0^t \{k_{03}\ex S_1\uN(u\wedge\tiNCd) + Nk_{05}\}\,du. 
\ens
Using~\Ref{S1C-ex-bnd}, we thus find that
\eq\label{M0-1}
   \ex\Bl \{M_0\uN(t\wedge\tiNCd)\}^2 \Br 
     \Le \frac{k_{03}}{k_{11}}\, N(C_1 + k_{14}T)(e^{k_{11}t}-1) + Nk_{05}t, 
\en
uniformly for all~$C'$.%
\ignore{
Hence $\{M_0\uN(t\wtxnn),\, n\ge0\}$, is a square integrable
martingale, implying that the limit $M_0\uN(t\wtxin) := 
\lnti M_0\uN(t\wtxnn)$ is well
defined, and satisfies $\ex M_0\uN(t\wtxin) = 0$.  Then, in view of the
Markov property of~$X_N$, it also follows that, for any $s\le t$,
$$
  \ex \Blb M_0\uN(t\wtxin) \giv \ff\uN_s \Brb = M_0\uN(s\wtxin),
$$
where, as before, $\ff\uN_s := \s\{X_N(u),\, 0\le u\le s\}$, so that 
\hbox{$\{M_0\uN(s\wtxin),\, 0\le s \le t\}$} 
is also a square integrable martingale.  This in turn,
by the martingale convergence theorem, implies from~\Ref{S0-bnd}
that
\eq\label{S0-txi-bnd}
  \sup_{0\le u < t\wtxin} S_0\uN(u) \ <\ \infty
\en
for any~$t$.
Note also that $M_0\uN(\cdot)$ is linear on any interval $[\txn,t^X_{n+1})$,
from which it follows that
$$
  \sup_{0\le u\le T} M_0\uN(u\wtxi) \Eq \sup_{0\le n\le \infty} M_0\uN(T\wtxn).
$$ 
} 
 Doob's maximal inequality applied to $M_0\uN(t\wedge\tiNCd)$ now allows us to deduce that, 
for any $C',a>0$,
\eqs
   \lefteqn{\pr\Bigl[\sup_{0\le u\le T} M_0\uN(u\wedge\tiNCd) > aN \Bigr]}\\
      &&\qquad\Le \frac1{Na^2}\Blb\frac{k_{03}}{k_{11}}\,(C_1 + k_{14}T)\{e^{k_{11}T}-1\} 
   + k_{05}T\Brb \ =:\ \frac{C_1K_{01} + K_{02}}{Na^2},
\ens
say, so that, letting $C'\to\infty$, 
\[
   \pr\Bigl[\sup_{0\le u\le T} M_0\uN(u) > aN \Bigr] \Le \frac{C_1K_{01} + K_{02}}{Na^2}
\]
also.
Taking $a = \half Ce^{-k_{01}T}$ and putting the result into~\Ref{S0-bnd},
the lemma follows.
\ep

\bigskip
\ignore{
As a consequence, $S_0\uN(t) = \|X_N(t)\|_1$ is almost surely bounded 
up to~$\txin$ if $\txin < \infty$.  This does not of 
itself imply that the event $\txin < \infty$
has zero probability, because the transition rates out of
sets of states~$X$ with~$\|X\|_1$ bounded may still be unbounded.
}
In the next theorem, we 
control the `higher $\n$-moments' $S_r\uN(t)$ of~$X_N(t)$.  

\begin{theorem}\label{AP-2}
Assume that Assumptions~\ref{ap-assns} are satisfied, and  that 
$S_1\uN(0) \le NC_1$ and $S_{p(1)}\uN(0) \le NC'_{1}$. 
Then, for $2\le r \le \rmax\ui$ and for any $C>0$, we have
\eq\label{Sr-mean}
   \ex\{S_r\uN(t\wedge\t_0\uN(C))\} 
     \Le (S_r\uN(0) + Nk_{r4}t) e^{(k_{r1} + Ck_{r2})t},  \quad 0 \le t\le T.
\en

Furthermore, if for $1\le r\le \rmax\ut$, $S_r\uN(0) \le NC_r$ and $S_{p(r)}\uN(0) \le NC'_{r}$, 
then, for any $\gga \ge 1$,
\eq\label{Sr-prob}
   \pr\bigl[\sup_{0\le t\le T} S_r\uN(t\wedge\t_0\uN(C))
             \ge N\gga C''_{rT}\bigr]
    \Le K_{r0}\gga^{-2}N^{-1},
\en
where 
\[
   C''_{rT} \Def (C_r + k_{r4}T + \sqrt{(C'_{r}\vee1)}) e^{(k_{r1} + Ck_{r2})T}
\]
and $K_{r0}$ depends on $C,T$ and the parameters of the model.
\end{theorem}

\bp
Recalling~\Ref{M0-def}, use the argument leading to~\Ref{S1C-ex-bnd} with
the martingales $M_r\uN(t\wedge\tiNCd\wedge\t_0\uN(C))$,
for any $C'>0$, to deduce that
\eqa
   \lefteqn{\ex S_r\uN(t\wedge\tiNCd\wedge\t_0\uN(C))} \non\\
   &\le& S_r\uN(0) + \int_0^t \Bl \{k_{r1} + C k_{r2}\}
          \ex\Blb S_r\uN(u\wedge\tiNCd\wedge\t_0\uN(C))\Brb + Nk_{r4}\Br \,du ,
   \non
\ena
for $1 \le r \le r_{\max}^{(1)}$, since $N^{-1}S_0\uN(u) \le C$ when $u\le \t_0\uN(C)$:
define $k_{12} = 0$. 
Gronwall's inequality now implies that
\eq
  \label{ESrn-bnd}
  \ex S_r\uN(t\wedge\tiNCd \wedge\t_0\uN(C)) \Le (S_r\uN(0) + Nk_{r4}t)e^{(k_{r1} + Ck_{r2})t},
\en
for $1\le r\le \rmax\ui$,
and \Ref{Sr-mean} follows by Fatou's lemma, on letting $C'\to\infty$.

Now, also from \Ref{M0-def} and~\Ref{U-bnd}, we have, for $t\ge0$ and each $r\le \rmax\ui$,
\eqa
   \lefteqn{S_r\uN(t\wedge\t_0\uN(C))} \non\\
    &=& S_r\uN(0)  + N\int_0^{t\wedge\t_0\uN(C)} U_r\uN(u) \,du 
     + M_r\uN(t\wedge\t_0\uN(C)) \non\\
     &\le&  S_r\uN(0) 
      + \int_0^t \Bl \{k_{r1} + Ck_{r2}\}S_r\uN(u\wedge\t_0\uN(C))+Nk_{r4}\Br \,du \non\\
    &&\qquad\mbox{} + \sup_{0\le u\le t} M_r\uN(u\wedge\t_0\uN(C)).  \non
\ena
Hence, from Gronwall's inequality, for all $t\ge0$ and $r\le \rmax\ui$, 
\eq\label{Sr-bnd}
   S_r\uN(t\wedge\t_0\uN(C)) \Le \Bigl\{N(C_r + k_{r4}t) + \sup_{0\le u\le t}
       M_r\uN(u\wedge\t_0\uN(C))\Bigr\}e^{(k_{r1} + Ck_{r2})t}.
\en

Now, as in~\Ref{variance-MG}, we have
\eq\label{variance-MG-r}
   \ex\Blb \{M_r\uN(t\wedge\tiNCd\wedge\t_0\uN(C))\}^2 -
      N\int_0^{t\wedge\tiNCd\wedge\t_0\uN(C)} V_r\uN(u)\,du \Brb
     \Eq 0,
\en
from which it follows, using~(\ref{V-bnd}), that, for $1 \le r \le r_{\max}^{(2)}$,
\eqs
   \lefteqn{\ex\Bl \{M_r\uN(t\wedge\tiNCd\wedge\t_0\uN(C))\}^2 \Br} \\
      &\le& \ex\Blb N\int_0^{t\wedge\tiNCd\wedge\t_0\uN(C))} V_r\uN(u)\,du \Brb\\
      &\le& \int_0^t \{k_{r3}\ex S_{p(r)}\uN(u\wedge\tiNCd\wedge\t_0\uN(C)) + Nk_{r5}\}\,du \\
      &\le& \frac{N(C'_r + k_{p(r),4}T)k_{r3}}{k_{p(r),1}+Ck_{p(r),2}}\, 
           (e^{(k_{p(r),1}+Ck_{p(r),2}t)}-1) + Nk_{r5}T, 
\ens
this last by~\Ref{ESrn-bnd}, since $p(r) \le r^{(1)}_{\max}$ for $1 \le r \le r^{(2)}_{\max}$.  
Using Doob's inequality, it follows that, for any~$a>0$,
\eqs
   \lefteqn{\pr\Bigl[\sup_{0\le u\le T} M_r\uN(u\wedge\t_0\uN(C)) > aN \Bigr]} \\
      &\le& \frac1{Na^2}\Blb\frac{k_{r3}(C'_r + k_{p(r),4}T)}{k_{p(r),1}+Ck_{p(r),2}}\, 
           (e^{(k_{p(r),1}+Ck_{p(r),2}T)}-1) + k_{r5}T\Brb \\ 
      &=:&  \frac{C'_rK_{r1} + K_{r2}}{Na^2}\,.
\ens
Taking $a = \gga\sqrt{(C'_r\vee1)}$ and putting the result into~\Ref{Sr-bnd} 
gives~\Ref{Sr-prob}, with $K_{r0} = (C'_r K_{r1} + K_{r2})/(C'_r \lor 1)$.
\ep

\nin Note also that $\sup_{0\le t\le T}S_r\uN(t) < \infty$ a.s.\ for 
all~$0 \le r \le \rmax\ut$, in view of Lemma~\ref{AP-1} and Theorem~\ref{AP-2}.

\bigskip
In what follows, we shall particularly need to control quantities of the form
$\sjj \a_J(x_N(s))d(J,\z)$,  where $x_N := N^{-1}X_N$ and 
\eq\label{djz-def}
   d(J,\z) \Def \sjo |J^j|\z(j),
\en
for $\z\in\rr$ chosen such that~$\z(j)\ge1$ grows fast enough with~$j$:
see~\Ref{zeta-cond}.  
Defining 
\eq\label{tau-az}
   \t\uN(a,\z) \Def \inf\Blb s\colon\, \sjj \a_J(x_N(s))d(J,\z) \ge a\Brb,
\en
infinite if there is no such~$s$, 
we show in the following corollary that, under suitable assumptions,
$\t\uN(a,\z)$ is rarely less than~$T$.

\begin{corollary}\label{tau-az-bnd}
Assume that Assumptions~\ref{ap-assns} hold, and that~$\z$ is such that
\eq\label{zeta-bnd}
   \sjj \a_J(N^{-1}X)d(J,\z) \Le \{k_1N^{-1}S_{r}(X) + k_2\}^{\bbt}
\en
for some $1\le r := r(\z) \le \rmax\ut$ and some $\bbt = \bbt(\z) \ge1$.
For this value of $r$, assume that  $S_{r}\uN(0) \le NC_{r}$ and 
$S_{p(r)}\uN(0) \le NC'_{r}$ for some constants $C_{r}$ and~$C'_{r}$.
Assume further that $S_0\uN(0) \le NC_0$, 
$S_1\uN(0) \le NC_1$ for some constants $C_0$, $C_1$, and define $C := 2(C_0 + k_{04}T)e^{k_{01}T}$.  Then
\[
   \pr[\t\uN(a,\z) \le T] \Le N^{-1}\{K_{r0}\gga_a^{-2} 
            + K_{00}(C_1\vee1)C^{-2}\},
\]
for any $a \ge \{k_2 + k_1 C''_{rT}\}^\bbt$,
where $\gga_a := (a^{1/\bbt} - k_2)/\{ k_{1}C''_{rT}\}$,  $K_{r0}$ and~$C''_{rT}$
are as in Theorem~\ref{AP-2}, and $K_{00}$ is as in Lemma~\ref{AP-1}. 
\end{corollary}

\bp
In view of~\Ref{zeta-bnd}, it is enough to bound the probability
$$
   \pr[\sup_{0\le t\le T} S_r\uN(t) \ge N(a^{1/\bbt}-k_2)/k_1].
$$  
However,
Lemma~\ref{AP-1} and Theorem~\ref{AP-2} together bound this probability
by 
$$
  N^{-1}\Blb K_{r0}\gga_a^{-2} + K_{00}(C_1\vee1)C^{-2} \Brb,
$$
where~$\gga_a$ is as defined above, as long as $a^{1/\bbt}-k_2 \ge k_1 C''_{rT}$.
\ep

\medskip
If~\Ref{zeta-bnd} is satisfied, $\sjj \a_J(x_N(s))d(J,\z)$ is a.s.\ bounded on 
$0 \le s\le T$, because~$S_r\uN(s)$ is.  The corollary shows that the sum is then 
bounded by $\{k_2 + k_{1}C''_{r,T}\}^{\bbt}$, except on an event of probability 
of order~$O(N^{-1})$. Usually, one can choose $\bbt=1$.

\section{Semigroup properties}\label{semigroup-section}
\setcounter{equation}{0}
We make the following initial assumptions about the matrix~$A$: first,
that
\eq
   A_{ij} \ge 0 \ \mbox{for all}\ i \neq j \ge 0;\qquad
      \sjni A_{ji} < \infty\  \mbox{for all}\ i\ge0, \label{A-assn-1} 
\en
and then that, for some $\m \in \R_+^{\Zp}$ such that $\m(m) \ge 1$ for each
$m\ge0$, and for some $w\ge0$,
\eq\label{A-assn-2}
   A^T \m \Le w\m.
\en
We then use~$\m$ to define the $\m$-norm 
\eq\label{Rmu}
    \nm{\x} \Def \smo \m(m)|\x^m| \quad\mbox{on}\quad
    \rr_\m \Def \{\x\in\rr\colon\,\nm{\x} < \infty\}.
\en
Note that there may be many possible choices for~$\m$.  In what follows,
it is important that~$F$ be a Lipschitz operator with respect to the $\m$-norm,
and this has to be borne in mind when choosing~$\m$.

Setting 
\eq\label{Q-def}
   Q_{ij} \Def A^T_{ij}\m(j)/\m(i) - w\d_{ij},
\en
where~$\d$ is the Kronecker delta, we note that 
$Q_{ij} \ge 0$ for $i\ne j$, and that
\[
   0 \Le  \sum_{j\ne i} Q_{ij} \Eq \sum_{j\ne i} A^T_{ij}\m(j)/\m(i) 
              \Le w-A_{ii} \Eq -Q_{ii},
\]
using~\Ref{A-assn-2} for the inequality, so that $Q_{ii}\le 0$. 
Hence~$Q$ can be augmented to a
conservative $Q$--matrix, in the sense of Markov jump processes, by adding a coffin
state~$\cof$, and setting $Q_{i\cof} := - \sum_{j\ge0} Q_{ij} \ge 0$.  Let~$P(\cdot)$
denote the semigroup of Markov transition matrices corresponding to the minimal
process associated with~$Q$; then, in particular,
\eq\label{P-properties}
  Q \Eq P'(0) \quad\mbox{and}\quad P'(t) \Eq QP(t) \quad\mbox{for all}\ t\ge0
\en
(Reuter~1957, Theorem~3). 
Set 
\eq\label{R-def}
   R^T_{ij}(t) \Def e^{wt}\m(i)P_{ij}(t)/\m(j).
\en

\begin{theorem}\label{semigroup}
Let~$A$ satisfy Assumptions \Ref{A-assn-1} and~\Ref{A-assn-2}.  Then, with the
above definitions, $R$ is a strongly continuous semigroup on~$\rr_\m$, and
\eq\label{R-inequality}
   \sio \m(i)R_{ij}(t) \Le \m(j)e^{wt} \qquad \mbox{for all}\ j\ \mbox{and}\ t.
\en
Furthermore, the sums $\sjo R_{ij}(t)A_{jk} = (R(t)A)_{ik}$ are well defined
for all $i,k$, and 
\eq\label{A-properties}
  A \Eq R'(0) \quad\mbox{and}\quad R'(t) \Eq R(t)A \quad\mbox{for all}\ t\ge0.
\en
\end{theorem}

\bp
We note first that, for $x\in\rr_\m$, 
\eqa
    \nm{R(t)x} &\le& \sio\m(i)\sjo R_{ij}(t)|x_j| \Eq e^{wt} \sio\sjo\m(j)P_{ji}(t)|x_j|\non\\
      &\le& e^{wt} \sjo\m(j)|x_j| \Eq e^{wt}\nm{x},\label{xiR-bnd}
\ena
since~$P(t)$ is substochastic on $\Zp$; hence $R\colon\,\rr_\m \to \rr_\m$. 
To show strong continuity, we take $x\in\rr_\m$, and consider 
\eqs
    \lefteqn{\nm{R(t)x - x} \Eq \sio\m(i)\Blm \sjo R_{ij}(t)x_j - x_i\Brm
     \Eq \sio\Blm e^{wt}\sjo \m(j)P_{ji}(t)x_j - \m(i)x_i \Brm} \\
  &\le& (e^{wt}-1)\sio\sjo \m(j)P_{ji}(t)x_j + \sio\sum_{j\ne i}\m(j)P_{ji}(t)x_j  
    + \sio \m(i) x_i(1-P_{ii}(t)) \\
  &\le& (e^{wt}-1)\sjo\m(j)x_j + 2\sio \m(i)x_i(1-P_{ii}(t)),
\ens
from which it follows that $\lim_{t\to0}\nm{R(t)x - x} = 0$, by dominated
convergence, since $\lim_{t\to0}P_{ii}(t) = 1$ for each~$i\ge0$.

The inequality~\Ref{R-inequality} follows from the definition of~$R$ and
the fact that~$P$ is substochastic on~$\Zp$.  Then
\eqs
  (A^TR^T(t))_{ij} &=& \sum_{k\ne i} Q_{ik} \frac{\m(i)}{\m(k)}\,e^{wt}
    \frac{\m(k)}{\m(j)} P_{kj}(t) + (Q_{ii}+w)e^{wt} \frac{\m(i)}{\m(j)} P_{ij}(t) \\
  &=& \frac{\m(i)}{\m(j)}\,\left[(QP(t))_{ij} + wP_{ij}(t)\right]e^{wt},
\ens
with $(QP(t))_{ij} = \sko Q_{ik}P_{kj}(t)$ well defined because~$P(t)$ is 
sub-stochastic and~$Q$ is conservative.  Using~\Ref{P-properties}, this gives
\[
  (A^TR^T(t))_{ij}  \Eq \frac{\m(i)}{\m(j)}\,\frac{d}{dt}[P_{ij}(t)e^{wt}] 
              \Eq  \frac{d}{dt} R^T_{ij}(t),
\]
and this establishes~\Ref{A-properties}.
\ep

\section{Main approximation}\label{approximation}
\setcounter{equation}{0}
Let $X_N$, $N\ge1$, be a sequence of pure jump Markov processes as in 
Section~\ref{introduction}, with $A$ and~$F$ defined as in \Ref{determ}
and~\Ref{F-assn}, and suppose that~$F\colon\,\rr_\m \to \rr_\m$, with~$\rr_\m$
as defined in~\Ref{Rmu}, for some~$\m$
such that Assumption~\Ref{A-assn-2} holds.  Suppose also that~$F$ is locally Lipschitz  
in the $\m$-norm: for any $z>0$,
\eq\label{F-assn-2}
  \sup_{x \neq y\colon\,\nm{x},\nm{y} \le z}\nm{F(x)-F(y)}/\nm{x-y} \Le K(\m,F;z)
     \ <\ \infty.
\en 
Then, for~$x(0) \in \rr_\m$ and~$R$ as in~\Ref{R-def}, the integral equation 
\eq\label{deterministic}
   x(t) \Eq R(t) x(0) + \int_0^t R(t-s)F(x(s))\,ds.
\en
has a unique continuous solution~$x$ in~$\rr_\m$ on some non-empty time
interval $[0,\tmax)$, such that, if $\tmax < \infty$, then 
$\|x(t)\|_{\mu} \to \infty$ as $t \to \tmax$ (Pazy~1983, Theorem~1.4, Chapter~6). 
Thus, if~$A$ were the generator of~$R$, the function~$x$ would be  
a {\it mild solution\/} of the deterministic equations~\Ref{determ}.
We now wish to show that the process $x_N := N^{-1}X_N$ is close to~$x$.
To do so, we need a corresponding representation for~$X_N$.

To find such a representation, let $W(t)$, $t\ge0$, be a pure jump path 
on~$\xx_+$ that has 
only finitely many jumps up to time~$T$. Then we can write
\eq\label{1}
   W(t) \Eq W(0) + \ssjt \D W(\s_j),\qquad 0\le t\le T,
\en
where $\D W(s) := W(s) - W(s-)$ and $\s_j$, $j\ge1$, denote the times 
when~$W$ has its jumps.  
Now let~$A$ satisfy \Ref{A-assn-1} and~\Ref{A-assn-2}, 
and let~$R(\cdot)$ be the associated semigroup, as defined in~\Ref{R-def}.
Define the path $W^*(t)$, $0\le t\le T$, from the equation
\eq
\begin{array}{l}
  W^*(t) \DEq R(t)W(0) +  \ssjt R(t-\s_j)\D_j 
  - \int_0^t R(t-s)A W(s)\,ds,\qquad \label{1.5}
\end{array}
\en
where $\D_j := \D W(\s_j)$. 
Note that the latter integral makes sense, because each of the sums
$\sjo R_{ij}(t)A_{jk}$ is well defined, from Theorem~\ref{semigroup}, 
and because only finitely many of the coordinates of~$W$ are non-zero.

\begin{lemma}\label{X-star=X}
$W^*(t) = W(t)$ for all~$0 \le t \le T$.
\end{lemma}

\begin{proof}
Fix any~$t$, and suppose that $W^*(s) = W(s)$ for all~$s \le t$.
This is clearly the case for $t=0$.
Let~$\s(t)>t$ denote the time of the first jump of~$W$
after~$t$.  Then, for any $0 < h < \s(t) - t$,  using the semigroup property
for~$R$ and~\Ref{1.5},
\eqa
  \lefteqn{W^*(t+h) - W^*(t)} \non \\
    &&=\ (R(h)-I)R(t)W(0)
            + \ssjt (R(h)-I)R(t-\s_j)\D_j \label{1.7} \\
    &&\qquad \mbox{} -  \int_0^t (R(h)-I)R(t-s)A W(s)\,ds
        -  \int_t^{t+h} R(t+h-s)A W(t)\,ds, \non
\ena
where, in the last integral, we use the fact that there are no jumps of~$W$
between $t$ and~$t+h$. Thus we have
\eqa
   \lefteqn{W^*(t+h) - W^*(t)}\non \\
        &=& (R(h)-I)\left\{R(t)W(0) + \ssjt R(t-\s_j)\D_j
       -  \int_0^t R(t-s)A W(s)\,ds\right\} \non \\
    &&\qquad \mbox{} -  \int_t^{t+h} R(t+h-s)A W(t)\,ds \non \\
    &=& (R(h)-I)W(t) - \int_t^{t+h} R(t+h-s)A W(t)\,ds.  \label{1.8}
\ena
But now, for~$x\in\xx_+$,
\[
    \int_t^{t+h} R(t+h-s)A x \,ds \Eq (R(h)-I)x,
\]
from~\Ref{A-properties}, so that $W^*(t+h) = W^*(t)$ for all $t+h < \s(t)$, implying
that $W^*(s) = W(s)$ for all~$s < \s(t)$.  On the
other hand,  from~\Ref{1.5}, we have $W^*(\s(t)) - W^*(\s(t)-) = \D W(\s(t))$, so that
$W^*(s) = W(s)$ for all~$s \le \s(t)$. Thus we can prove equality over the
interval $[0,\s_1]$, and then successively over the intervals $[\s_j,\s_{j+1}]$,
until $[0,T]$ is covered.
\end{proof}

\bigskip
Now suppose that~$W$ arises as a realization of~$X_N$.  
Then~$X_N$ has transition rates such that 
\eq\label{2}
   M_N(t) \DEq  \ssjt \D X_N(\s_j) 
       -  \int_0^t A X_N(s)\,ds - \int_0^t NF(x_N(s))\,ds
\en
is a zero mean local martingale.  In view
of Lemma~\ref{X-star=X}, we can use~\Ref{1.5} to write
\eq\label{3}
   X_N(t) \Eq R(t)X_N(0) +  \NNN(t) + N\int_0^t R(t-s)F(x_N(s))\,ds,
\en
where
\eqa
     \lefteqn{\NNN(t) \DEq    \ssjt R(t-\s_j)\D X_N(\s_j)} \non\\
     &&\mbox{}\qquad-  \int_0^t R(t-s)A X_N(s)\,ds
       -   \int_0^t R(t-s)NF(x_N(s))\,ds.\phantom{HHH}\label{N-def}
\ena

Thus, comparing~\Ref{3} and~\Ref{deterministic}, we expect
$x_N$ and~$x$ to be close, for $ 0\le t\le T < \tmax$,
provided that we can show that $\sup_{t\le T}\nm{\nnn(t)}$ is small,
where $\nnn(t) := N^{-1}\NNN(t)$. Indeed,
if $x_N(0)$ and~$x(0)$ are close, then
\eqa
   \lefteqn{\nm{x_N(t) - x(t)} }\non\\[1ex]
       &&\le\ \nm{ R(t)(x_N(0) - x(0)) }  \non\\
   &&\qquad\mbox{}+
       \int_0^t \nm{ R(t-s)[F(x_N(s)) - F(x(s))] }\,ds
       + \nm{\nnn(t)} \phantom{HHHH}\non\\
   &&\Le  e^{wt} \nm{x_N(0) - x(0) } \non\\
   &&\qquad\mbox{} +
      \int_0^t e^{w(t-s)} K(\m,F;2\Xi_T)\nm{x_N(s) - x(s)}\,ds 
        + \nm{\nnn(t)},\label{main-ineq}
\ena
by~\Ref{xiR-bnd}, with the stage apparently set for Gronwall's inequality,
assuming that $\nm{x_N(0) - x(0) }$ and $\sutt\nm{\nnn(t)}$ are small
enough that then $\nm{x_N(t)} \le 2\Xi_T$ for $0\le t\le T$, where
$\Xi_T := \sup_{0\le t\le T}\nm{x(t)}$. 

Bounding $\sutt\nm{\nnn(t)}$ is, however, not so easy.
Since~$\NNN$ is not itself a martingale, we cannot directly
apply martingale inequalities to control its fluctuations.
However, since
\eq\label{tmn-expression}
   \NNN(t) \Eq \int_0^{t} R(t-s)\,dM_N(s),
\en
we can hope to use control over the local martingale~$M_N$ instead.  
For this and the subsequent argument,
we introduce some further assumptions.  

\begin{assumption}\label{zeta-assns}\mbox{}\\[2ex]
1.\quad There exists $r=r_\m \le r_{\max}^{(2)}$ such that 
   $\sup_{j\ge0}\{\m(j)/\n_r(j)\} < \infty$.\\[2ex]
2.\quad There exists $\z\in\rr$ with $\z(j)\ge1$ for all~$j$ such 
that~\Ref{zeta-bnd} is 
satisfied for some $\bbt = \bbt(\z) \ge 1$ and $r=r(\z)$ such that 
  $1\le r(\z)\le \rmax\ut$, and that
\eq\label{zeta-cond}
   \ZZZ \Def \sko \frac{\m(k)(|A_{kk}|+1)}{\sqrt{\z(k)}} \ <\ \infty.
\en
\end{assumption}

\nin  The requirement that~$\z$ satisfies~\Ref{zeta-cond} as well as
satisfying~\Ref{zeta-bnd} for some $r \le \rmax\ut$ implies in practice
that it must be possible to take $\rmax\ui$ and~$\rmax\ut$ to be quite
large in Assumption~\ref{ap-assns}; see the examples in Section~\ref{examples}. 

\medskip
Note that part 1 of Assumption~\ref{zeta-assns} implies that 
$\lim_{j\to\infty}\{\m(j)/\n_r(j)\} = 0$ for some
$r = \tr_\m \le r_\m+1$.  We define
\eq\label{rho-def}
   \r(\z,\m) \Def \max\{r(\z),p(r(\z)),\tr_\m\},
\en
where~$p(\cdot)$ is as in Assumptions~\ref{ap-assns}.
We can now prove the following lemma, which enables us to control the paths
of~$\NNN$ by using fluctuation bounds for the martingale~$M_N$.

\begin{lemma}\label{fubini}
Under Assumption~\ref{zeta-assns},
\[
   \NNN(t) \Eq M_N(t) + \int_0^t R(t-s)AM_N(s)\,ds.
\]
\end{lemma}

\bp
From~\Ref{A-properties}, we have
\[
   R(t-s) \Eq I + \int_0^{t-s} R(v)A\,dv.
\]
Substituting this into~\Ref{tmn-expression}, we obtain
\eqs
  \NNN(t) &=& \int_0^{t} R(t-s)\,dM_N(s) \\
   &=& M_N(t) + \int_0^t\Blb \int_0^t R(v)A\bone_{[0,t-s]}(v)\,dv\Brb
     dM_N(s) \\
   &=& M_N(t) + \int_0^t\Blb \int_0^t R(v)A\bone_{[0,t-s]}(v)\,dv\Brb dX_N(s) \\
  &&\qquad\mbox{}- \int_0^t\Blb \int_0^t R(v)A\bone_{[0,t-s]}(v)\,dv\Brb F_0(x_N(s))\,ds.
\ens
It remains to change the order of integration in the double integrals, 
for which we use Fubini's theorem.

In the first, the outer integral is almost surely a finite sum, and at each
jump time~$t^{X_N}_l$ we have $dX_N(t^{X_N}_l) \in \jj$.  Hence it is enough that, for
each $i$, $m$ and~$t$, $\sjo R_{ij}(t)A_{jm}$ is absolutely summable, which 
follows from Theorem~\ref{semigroup}.   Thus we have
\eq\label{swap-1}
   \int_0^t\Blb \int_0^t R(v)A\bone_{[0,t-s]}(v)\,dv\Brb dX_N(s)
    \Eq \int_0^t R(v)A\{X_N(t-v) - X_N(0)\}\,dv.
\en

For the second, the $k$-th component of $R(v)A F_0(x_N(s))$ is just
\eq\label{integrand-2}
    \sjo R_{kj}(v) \slo A_{jl} \sjj J^l \a_J(x_N(s)).
\en
Now, from~\Ref{R-inequality}, we have $0\le R_{kj}(v) \le \m(j)e^{wv}/\m(k)$,
and 
\eq\label{m-A-sum}
     \sjo \m(j)|A_{jl}| \Le \m(l)(2|A_{ll}|+w),
\en
because $A^T\m \le w\m$.  Hence, putting absolute values in the 
summands in~\Ref{integrand-2} yields at most
\[
     \frac{e^{wv}}{\m(k)}\,\sjj\a_J(x_N(s))\,\slo |J^l|\m(l)(2|A_{ll}|+w).
\]
Now, in view of~\Ref{zeta-cond} and since $\z(j)\ge1$ for all~$j$, 
there is a constant $K<\infty$ such that $\m(l)(2|A_{ll}|+w) \le K\z(l)$. 
Furthermore, $\z$ satisfies~\Ref{zeta-bnd}, so that, by Corollary~\ref{tau-az-bnd},
$\sjj\a_J(x_N(s))\,\slo |J^l|\z(l)$ is a.s.\ uniformly bounded 
in $0 \le s \le T$.  Hence we can apply Fubini's theorem,
obtaining
\[
   \int_0^t\Blb \int_0^t R(v)A\bone_{[0,t-s]}(v)\,dv\Brb F_0(x_N(s))\,ds
    \Eq \int_0^t R(v)A \Blb\int_0^{t-v} F_0(x_N(s))\,ds\Brb\,dv,
\]
and combining this with~\Ref{swap-1} proves the lemma.
\ep

\bigskip
We now introduce the exponential martingales that we use to bound
the fluctuations of~$M_N$.
For~$\th\in\R^{\Z_+}$ bounded and $x \in \rr_\m$, 
\[
   Z_{N,\th}(t) \Def  \begin{array}{ll}  e^{\th^T x_N(t)}\exp\Blb 
        - \int_0^t g_{N\th}(x_N(s-)) \,ds \Brb,  & t \ge 0,
             \end{array}
\]
is a non-negative finite variation local martingale, where
\[
   g_{N\th}(\xi) \Def \sjj N\a_J(\xi)\Bl e^{N^{-1}\th^T J} - 1 \Br.
\]
For $t\ge 0$, we have
\eqa
   \log Z_{N,\th}(t)
    &=& \th^T x_N(t)  - \int_0^t g_{N\th}(x_N(s-))\,ds    \non\\         
    &=&  \th^T m_N(t) - \int_0^t \f_{N,\th}(x_N(s-),s)\,ds,
    \label{Z-expression}
\ena
where
\eq\label{phi-def}
  \f_{N,\th}(\xi) \Def \sjj N\a_J(\xi)\Bl e^{N^{-1}\th^T J} - 1
          - N^{-1}\th^T J \Br,
\en
and $m_N(t) := N^{-1}M_N(t)$.
Note also that we can write
\eq\label{phi-def-2}
    \f_{N,\th}(\xi) \Eq N\int_0^1 (1-r) D^2 v_N(\xi,r\th)[\th,\th]\,dr,
\en
where
\[
      v_N(\xi,\th') \Def \sjj \a_J(\xi) e^{N^{-1}(\th')^T J},
\]
and $D^2 v_N$ denotes the matrix of second derivatives with respect to the second argument:
\eq\label{D2-explicit}
     D^2 v_N(\xi,\th')[\z_1,\z_2] \Def N^{-2}\sjj \a_J(\xi) e^{N^{-1}(\th')^T J}
        \z_1^TJJ^T \z_2
\en
for any $\z_1,\z_2\in\rr_\m$.

Now choose any $B := (B_k,\,k\ge0) \in \rr$, and define $\tilt_k\uN(B)$ by
$$
   \tilt_k\uN(B) \Def \inf \Blb t \ge 0\colon\, \sum_{J: J_k \not = 0} \a_J (x_N(t-)) > B_k\Brb.
$$
Our exponential bound is as follows.

\begin{lemma}\label{compt-bnd}
For any $k\ge 0$, 
$$
   \pr \left[ \sup_{0 \le t \le T\wtkb} |m_N^k (t)| \ge \delta \right] 
       \Le 2 \exp (-\delta^2N/2B_k K_*T).
$$
for all $0 < \delta \le  B_k K_*T$, where $K_* := J_*^2 e^{J_*}$, and~$J_*$ is as
in~\Ref{finite-jumps}.
\end{lemma}

\bp
Take $\th = e^{(k)}\beta$, for~$\beta$ to be chosen later.
We shall argue by stopping 
the local martingale $Z_{N,\th}$ at time $\s\uN(k,\d)$,
where
\[
   \s\uN(k,\d) \Def T \wedge \tkb \wedge\inf\{t\colon\, m_N^k(t) \ge \d\}.
\]
Note that $ e^{N^{-1}\th^T J} \le e^{J_*}$,
so long as $|\beta |\le N$, so that 
\eqs
   &&D^2 v_N(\xi,r\th)[\th,\th] 
      \Le N^{-2}\Bl\sum_{J: J_k \not = 0} \a_J(\xi)\Br \beta^2 K_*  .
\ens
Thus, from \Ref{phi-def-2}, 
we have
\[
      \f_{N,\th}(x_N(u-)) \Le \half  N^{-1} B_k  \beta^2 K_*,\qquad u \le \tkb,
\]
and hence, on the event that $\s\uN(k,\d) = \inf\{t\colon\, m_N^k(t) \ge \d\} 
\le (T \wedge \tkb)$, we have
\[
   Z_{N,\th}(\s(k,\d)) \Ge \exp\{\beta \d - \half N^{-1} B_k \beta^2 K_* T\}.
\]
But since $Z_{N,\th}(0) = 1$, it now follows from the optional stopping theorem and
Fatou's lemma that
\eqs
   1 &\ge& \ex\{Z_{N,\th}(\s\uN(k,\d))\} \\
     &\ge&
     \pr\Bigl[\sup_{0\le t\le T\wtkb} m_N^k(t) \ge \d\Bigr]
           \exp\{\beta \d - \half N^{-1} B_k \beta^2 K_* T\}.        
\ens
We can choose $\beta = \delta N/B_k K_* T$, 
as long as $\delta /  B_k K_* T \le 1$,  obtaining
$$
     \pr \Bl \sup_{0\le t\le T\wtkb} m_N^k (t) \ge \delta \Br
      \Le \exp (-\delta^2N/2B_k K_* T ).
$$
Repeating with 
\[
    \tilde\s\uN(k,\d) \Def T \wedge \tkb \wedge\inf\{t\colon\, -m_N^k(t) \ge \d\},  
\]
and choosing $\beta = \delta N/B_k K_* T$, gives the lemma.
\ep

\bigskip
The preceding lemma gives a bound for each individual component of~$M_N$.
We need first to translate this into a statement for all components simultaneously.
For~$\z$ as in Assumption~\ref{zeta-assns}, we start by writing
\eq
   Z_*\ui \Def \max_{k\ge1} k^{-1}\#\{m\colon\,\z(m) \le k\} ; \quad
   Z_*\ut \Def \sup_{k\ge0} \frac{\m(k)(|A_{kk}|+1)}{\sqrt{\z(k)}}\,.
\label{Zstar-defs}
\en
$Z_*\ut$ is clearly finite, because of Assumption~\ref{zeta-assns}, and the same is
true for~$Z_*\ui$ also, since~$Z$
of Assumption~\ref{zeta-assns} 
is at least $\#\{m\colon\,\z(m) \le k\}/{\sqrt k}$, 
for each~$k$. Then, using the definition~\Ref{tau-az} of~$\t\uN(a,\z)$, note that,
for every~$k$, 
\eq\label{h-bnd-1}
  \sum_{J\colon\,J^k\neq0} \a_J(x_N(t))h(k) \Le 
     \sum_{J\colon\,J^k\neq0} \frac{\a_J(x_N(t))h(k)d(J,\z)}{|J^k|\z(k)} 
       \Le \frac{a h(k)}{\z(k)},
\en
for any $t < \t\uN(a,\z)$ and any $h \in \rr$, 
and that, for any ${\mathcal K}\subseteq \integ_+$,
\eqa
  \sum_{k\in {\mathcal K}}\sum_{J\colon\,J^k\neq0} \a_J(x_N(t))h(k) &\le&
     \sum_{k\in {\mathcal K}}\sum_{J\colon\,J^k\neq0} 
                    \frac{\a_J(x_N(t))h(k)d(J,\z)}{|J^k|\z(k)}\non \\
   &\le& \frac{a }{\min_{k\in {\mathcal K}}(\z(k)/h(k))}.\label{h-bnd-2}
\ena
From~\Ref{h-bnd-1} with $h(k)=1$ for all~$k$, if we choose
$B := (a/\z(k),\,k\ge0)$, then $\t\uN(a,\z) \le \tkb$ for all~$k$.
For this choice of~$B$, we can take
\eq\label{delta-def}
  \d^2_k \Def \d^2_k(a) \Def  \frac{4 a K_* T \log N}{N\z(k)} \Eq 
      \frac{4 B_k K_* T \log N}{N}  
\en
in Lemma~\ref{compt-bnd} for $ k\in \k_N(a)$, where
\eq\label{kappa-a-def}
   \k_N(a) \Def \left\{k\colon\, \z(k) \le \halfh a K_* T N/\log N\right\}
           \Eq \left\{k\colon\, B_k \ge 4\log N/ K_* T N\right\}, 
\en
since then $\d_k(a) \le B_kK_*T$. Note that then, from~\Ref{zeta-cond},
\eq\label{mu-norm-delta}
    \sum_{k\in\k_N(a)} \m(k)\d_k(a) \Le 2 \ZZZ \sqrt{a K_* T N^{-1}\log N},
\en
with $\ZZZ$ as defined in Assumption~\ref{zeta-assns}, and that 
\eq\label{mod-kappa}
   |\k_N(a)| \Le \halfh a Z_*\ui K_* T N/\log N.
\en

\begin{lemma}\label{mn-tilde-bits}
   If Assumptions~\ref{zeta-assns} are satisfied, taking $\d_k(a)$ and~$\k_N(a)$ as 
defined in \Ref{delta-def} and~\Ref{kappa-a-def}, and for any $\h\in\rr$, we have
\eqs
  1.&& \pr\left[ \bigcup_{k\in\k_N(a)} \Bigl\{\sup_{0\le t\le T\wtaz} |m_N(t)| 
     \ge \d_k(a) \Bigr\} \right] \Le \frac{aZ_*\ui K_* T}{ 2N\log N};\\
  2.&& \pr\left[ \sum_{k\notin\k_N(a)} X_N^k(t) = 0\ \mbox{for all}\ 0\le t\le T\wtaz\right]
     \ \ge\ 1 - \frac{4\log N}{K_*N}; \\
  3.&& \sup_{0\le t\le T\wtaz}\Blb \sum_{k\notin\k_N(a)} \h(k)|F^k(x_N(t))| \Brb
     \Le \frac{aJ_*}{\min_{k \notin \k_N(a)}(\z(k)/\h(k))}.
\ens
\end{lemma}

\bp
For part~1, use Lemma~\ref{compt-bnd} together with \Ref{delta-def} and~\Ref{mod-kappa}
to give the bound.
For part~2, the total rate of jumps into coordinates with indices $k \notin \k_N(a)$
is 
\[
   \sum_{k\notin\k_N(a)} \sum_{J\colon\,J^k\neq0} \a_J(x_N(t)) 
      \Le \frac{a }{\min_{k\notin\k_N(a)}\z(k)},
\]
if $t\le \t\uN(a,\z)$, using~\Ref{h-bnd-2} with $\mathcal K = (\k_N(a))^c$, which, combined with~\Ref{kappa-a-def}, 
proves the claim.  For the final part, if $t\le \t\uN(a,\z)$, 
\[
   \sum_{k\notin\k_N(a)}\h(k)|F^k(x_N(t))|  \Le
      \sum_{k\notin\k_N(a)}\h(k) \sum_{J\colon\,J^k\neq0} \a_J(x_N(t)) J_*,
\]
and the inequality follows once more from~\Ref{h-bnd-2}.
\ep

\bigskip

Let $B_N\ui(a)$ and~$B_N\ut(a)$ denote the events
\eqa
  B_N\ui(a) &:=& \Blb\sum_{k\notin\k_N(a)} X_N^k(t) = 0\ \mbox{for all}\ 
     0\le t\le T\wtaz\Brb;\non\\
  B_N\ut(a) &:=& \Bl\bigcap_{k\in\k_N(a)} \Bigl\{\sup_{0\le t\le T\wtaz} 
       |m_N(t)| 
     \le \d_k(a) \Bigr\}\Br, \label{BN-defs}
\ena
and set $B_N(a) := B_N\ui(a) \cap B_N\ut(a)$.
Then, by Lemma~\ref{mn-tilde-bits}, we deduce that
\eq\label{BN-complement}
   \pr[B_N(a)^c] \Le \frac{aZ_*\ui K_* T}{ 2N\log N} + \frac{4\log N}{K_*N},
\en
of order $O(N^{-1}\log N)$ for each fixed~$a$. 
Thus we have all the components of~$M_N$ simultaneously controlled, except
on a set of small probability.  We now translate this into the desired
assertion about the fluctuations of~$\nnn$.

\begin{lemma}\label{mN-tilde-bnd}
If Assumptions~\ref{zeta-assns} are satisfied, then, on the event~$B_N(a)$,
\[
   \suttwt \nm{\nnn(t)} \Le \sqrt a \,K_{\ref{mN-tilde-bnd}}\,\sflnn,
\]
where the constant $K_{\ref{mN-tilde-bnd}}$ depends on $T$ and the
parameters of the process.
\end{lemma}

\bp
From Lemma~\ref{fubini}, it follows that
\eqa
     \lefteqn{\suttwt\nm{\nnn(t)}} \label{nm-dissection}\\
    &\le& \suttwt\nm{m_N(t)}
       + \suttwt\int_0^t \nm{R(t-s)Am_N(s)}\,ds.\non
\ena
For the first term, on~$B_N(a)$ and for $0\le t\le T\wtaz$, we have
\[
     \nm{m_N(t)} \Le \sum_{k\in\k_N(a)} \m(k)\d_k(a)
     + \int_0^t \sum_{k\notin\k_N(a)} \m(k)|F^k(x_N(u))|\,du .
\]
The first sum is bounded using~\Ref{mu-norm-delta}
by $2 \ZZZ \sqrt{a K_* T}\,N^{-1/2}\sqrt{\log N}$, the second, 
from Lemma~\ref{mn-tilde-bits} and~\Ref{kappa-a-def}, 
by 
\[
   \frac{TaJ_*}{\min_{k \notin \k_N(a)}(\z(k)/\m(k))} \Le
     Z_*\ut 2J_*\sqrt{\frac{Ta}{K_*}}\,\sflnn .
\]

For the second term in~\Ref{nm-dissection}, from \Ref{R-inequality} 
and~\Ref{m-A-sum}, we note that
\eqs
   \nm{R(t-s)Am_N(s)} &\le& 
    \sko \m(k) \slo R_{kl}(t-s) \sro |A_{lr}| |m_N^r(s)| \\
   &\le& e^{w(t-s)}\slo \m(l) \sro |A_{lr}| |m_N^r(s)| \\
   &\le& e^{w(t-s)} \sro \m(r)\{2|A_{rr}|+w\} |m_N^r(s)|.
\ens
On~$B_N(a)$ and for $0\le s\le T\wtaz$, from~\Ref{zeta-cond}, the sum 
for~$r\in\k_N(a)$ is bounded using
\eqs
   \lefteqn{\sum_{r\in\k_N(a)} \m(r)\{2|A_{rr}|+w\} |m_N^r(s)|}\\
    &&\Le \sum_{r\in\k_N(a)} \m(r)\{2|A_{rr}|+w\}\d_r(a) \\
    &&\Le \sum_{r\in\k_N(a)} 
      \m(r)\{2|A_{rr}|+w\}\sqrt{\frac{4 a K_* T \log N}{N\z(r)}}\\
    &&\Le (2\vee w)\ZZZ \sqrt{4 a K_* T}\sflnn.
\ens
The remaining sum is then bounded by Lemma~\ref{mn-tilde-bits},
on the set~$B_N(a)$ and for $0\le s\le T\wtaz$, giving at most
\eqs
   \lefteqn{\sum_{r\notin\k_N(a)} \m(r)\{2|A_{rr}|+w\} |m_N^r(s)|}\\
    &&\Le \sum_{r\notin\k_N(a)} \m(r)\{2|A_{rr}|+w\}\int_0^s |F^r(x_N(t))|\,dt \\
   &&\Le \frac{(2\vee w)saJ_*}{\min_{k \notin \k_N(a)}(\z(k)/\m(k)\{|A_{kk}|+1\})}\\
   &&\Le (2\vee w)Z_*\ut 2J_*\sqrt{\frac{Ta}{K_*}}\,\sflnn\,.
\ens
Integrating, it follows that
\eqs
   \lefteqn{\suttwt\int_0^t \nm{R(t-s)Am_N(s)}\,ds}\\
   &\le& (2T\vee1)e^{wT}\Blb \sqrt{4 a K_* T}\ZZZ + 
     Z_*\ut J2J_*\sqrt{\frac{Ta}{K_*}}\Brb\,\sflnn,
\ens
and the lemma follows.
\ep
       
\bigskip
This has now established the control on~$\sutt \nm{\nnn(t)}$ that we need, in order
to translate~\Ref{main-ineq} into a proof of the main theorem.

\begin{theorem}\label{main-theorem}
Suppose that \Ref{finite-jumps}, \Ref{finite-rates}, \Ref{A-assn-1}, \Ref{A-assn-2}
and~\Ref{F-assn-2} are all satisfied, and that
Assumptions \ref{ap-assns}  and~\ref{zeta-assns} hold. Recalling the
definition~\Ref{rho-def} of~$\r(\z,\m)$,
for~$\z$ as given in Assumption~\ref{zeta-assns}, suppose that 
$S_{\r(\z,\m)}\uN(0) \le NC_*$ for some $C_* < \infty$.

Let~$x$ denote the solution to~\Ref{deterministic} with initial condition~$x(0)$
satisfying $S_{\r(\z,\m)}(x(0)) < \infty$.  Then $\tmax = \infty$.

Fix any $T$, and define $\Xi_T := \sup_{0\le t\le T}\nm{x(t)}$.
If $\nm{x_N(0) - x(0)} \le \half\Xi_T e^{-(w+k_*)T}$, where
$k_* := e^{wT}K(\m,F;2\Xi_T)$, then
there exist constants $c_1$, $c_2$ depending on $C_*$, $T$ and the parameters
of the process, such that for all $N$ large enough
\eqa
 \lefteqn{\pr \left(  \sutt\nm{x_N(t) - x(t)} 
      > \Bl e^{wT}\nm{x_N(0) - x(0)} + c_1\,\sflnn \Br e^{k_*T} \right)}\non\\
  && \qquad\Le \frac{c_2\log N}{N}.\phantom{HHHHHHHHHHHHHHHHHHHHHHHH} \label{main-ineq-1}
\ena
\end{theorem}

\bp
As $S_{\r(\z, \mu)}\uN(0) \le NC_*$, it follows also that 
$S_r\uN(0) \le NC_*$ for all $0\le r\le \r(\z,\m)$.  Fix any $T < \tmax$, take
$C := 2(C_*+k_{04}T)e^{k_{01}T}$, and observe that, 
for $r \le \r (\z, \mu) \wedge r_{\max}^{(2)}$, and such that $p(r) \le \r (\z, \mu)$, 
we can take
\eq\label{C''-bnd}
   C''_{rT} \Le \tC_{rT} \Def \{2(C_*\vee1) + k_{r4}T\}e^{(k_{r1}+Ck_{r2})T},
\en
in Theorem~\ref{AP-2}, 
since we can take $C_*$ to bound $C_r$ and~$C'_r$.
In particular, $r=r(\z)$ as defined in Assumption~\ref{zeta-assns} 
satisfies both the conditions on~$r$ for~\Ref{C''-bnd} to hold.  
Then, taking $a := \{k_2 + k_1\tC_{r(\z)T}\}^{\bbt(\z)}$
in Corollary~\ref{tau-az-bnd}, it follows that for some constant $c_3 > 0$, on the event $B_N(a)$,
\[
   \pr[\t\uN(a,\z) \le T] \Le c_3N^{-1}.
\]
Then, from~\Ref{BN-complement}, for some constant $c_4$, $\pr[B_N(a)^c] \le c_4 N^{-1}\log N$.
Here, the constants~$c_3,c_4$ depend on~$C_*$, $T$ and the parameters of the process.

We now use Lemma~\ref{mN-tilde-bnd} to bound the martingale term in~\Ref{main-ineq}. It follows that,
on the event~$B_N(a) \cap \{\t\uN(a,\z) > T\}$ and on the event that $\nm{x_N(s) - x(s)} \le \Xi_T$
for all $0\le s\le t$,
\eqs
   \nm{x_N(t) - x(t)} &\le& \Bl e^{wT}\nm{x_N(0) - x(0)} 
        + \sqrt a\,K_{\ref{mN-tilde-bnd}}\,\sflnn \Br \\
    &&\quad\mbox{} + k_*\int_0^t \nm{x_N(s) - x(s)}\,ds,
\ens
where $k_* := e^{wT}K(\m,F;2\Xi_T)$.  Then from Gronwall's inequality, on the 
event~$B_N(a) \cap \{\t\uN(a,\z) > T\}$,
\eq\label{Gronwall-ineq}
   \nm{x_N(t) - x(t)} 
      \Le \Bl e^{wT}\nm{x_N(0) - x(0)} + \sqrt a\,K_{\ref{mN-tilde-bnd}}\,\sflnn \Br
         e^{k^*t},
\en
for all $0\le t\le T$, provided that
\[
   \Bl e^{wT}\nm{x_N(0) - x(0)} + \sqrt a\,K_{\ref{mN-tilde-bnd}}\,\sflnn \Br
     \Le \Xi_T e^{-k^*T}.
\]
This is true for all~$N$ sufficiently large, if  
$\nm{x_N(0) - x(0)} \le \half\Xi_T e^{-(w+k^*)T}$, which we have assumed. We have thus proved~(\ref{main-ineq-1}), since, as shown above,  
$\pr (B_N(a)^c \cup \{\t\uN(a,\z) > T\}^c) = O(N^{-1}\log N)$.

We now use this to show that in fact $\tmax = \infty$.  For~$x(0)$ as above,
we can take $x_N^j(0) := N^{-1}\lfloor Nx^j(0) \rfloor \le x^j(0)$, so that
$S_{\r(\z,\m)}\uN(0) \le NC_*$ for $C_* := S_{\r(\z,\m)}(x(0)) < \infty$.
Then, by~(\ref{rho-def}), $\lim_{j\to\infty}\{\m(j)/\n_{\r(\z,\m)}(j)\} = 0$, so it follows
easily using bounded convergence that $\nm{x_N(0)-x(0)} \to 0$ as ${N\to\infty}$.  Hence, for any
$T < \tmax$, it follows from~\Ref{main-ineq-1}
that 
$\nm{x_N(t)-x(t)} \to_{D} 0$ as ${N\to\infty}$, for $t \le T$, with uniform bounds over 
the interval, where `$\to_{D}$' denotes convergence in distribution. 
Also, by Assumption~\ref{zeta-assns}, there is a constant $c_5$ such that 
$\nm{x_N(t)} \le c_5N^{-1}S_{r_\m}\uN(t)$ for each $t$, where $r_\m \le \rmax\ut$ 
and $r_{\mu} \le \rho (\zeta, \mu)$.
Hence, using Lemma~\ref{AP-1} and Theorem~\ref{AP-2},
$\sup_{0\le t\le 2T}\nm{x_N(t)}$ remains bounded in probability
as $N\to\infty$. 
Hence it is impossible that $\nm{x(t)} \to \infty$ as $T \to \tmax < \infty$,
implying that 
in fact $\tmax = \infty$ for such~$x(0)$.
\ep

\bigskip
\nin{\bf Remark}.\ The dependence on the initial conditions is considerably
complicated by the way the constant~$C$ appears in the exponent, for
instance in the expression for~$\tC_{rT}$ in the proof of Theorem~\ref{main-theorem}.  
However, if $k_{r2}$ in Assumptions~\ref{ap-assns} can be
chosen to be zero, as for instance in the examples below, the
dependence simplifies correspondingly. 

\bigskip
There are biologically plausible models in which the restriction to $J^l \ge -1$ is
irksome.  In populations in which members of a given type~$l$ can fight
one another,  a natural possibility is to have a transition $J = - 2e^{(l)}$  
at a rate proportional to $X^l(X^l-1)$, which translates to 
$\a_J = \a_J\uN = \g x^l(x^l - N^{-1})$, a function depending on~$N$. Replacing
this with $\a_J = \g (x^l)^2$ removes the $N$-dependence, but yields a process 
that can jump to negative values of~$X^l$.  For this reason, it is useful to be
able to allow the transition rates~$\a_J$ to depend on~$N$.

Since the arguments in this paper are not limiting arguments for $N\to\infty$,
it does not require many changes to derive the corresponding results.  
Quantities such as $A$, $F$, $U_r(x)$ and $V_r(x)$ now depend on~$N$; however,
Theorem~\ref{main-theorem} continues to hold with constants $c_1$ and~$c_2$
that do not depend on~$N$, provided that $\m$, $w$, $\n$, the $k_{lm}$ from 
Assumption~\ref{ap-assns} and $\z$ from Assumption~\ref{zeta-assns} can be chosen to
be independent of~$N$, and that the quantities $Z_*^{(l)}$ from~\Ref{Zstar-defs}
can be bounded uniformly in~$N$.  On the other hand, the solution~$x=x\uN$
of~\Ref{deterministic} that acts as approximation to~$x_N$ in 
Theorem~\ref{main-theorem} now itself depends on~$N$, through $R=R\uN$ and
$F=F\uN$.  If~$A$ (and hence~$R$) can be taken to be independent of~$N$, 
and $\lim_{N\to\infty}\nm{F\uN-F} = 0$ for some fixed $\m$--Lipschitz 
function~$F$, a Gronwall argument can be used to derive a bound for the 
difference between~$x\uN$ and the (fixed) solution~$x$ to equation~\Ref{deterministic} 
with $N$-independent $R$ and~$F$. If~$A$ has to depend on~$N$, the situation is more
delicate. 

\section{Examples}\label{examples}
\setcounter{equation}{0}
We begin with some general remarks, to show that the assumptions
are satisfied in many practical contexts.  We then discuss
two particular examples, those of Kretzschmar~(1993)
and of Arrigoni~(2003), that fitted poorly or not at all into the
general setting of Barbour \& Luczak~(2008), though the other systems
referred to in the introduction could also be treated similarly.  
In both of our chosen examples,
the index~$j$ represents a number of individuals ---  parasites in a
host in the first, animals in a patch in the second --- and we shall
for now use the former terminology for the preliminary, general
discussion. 

Transitions that
can typically be envisaged are: births of a few parasites, which may
occur either in the same host, or in another, if infection 
is being represented; births and immigration of hosts, with or without parasites;
migration of parasites between hosts; deaths of
parasites;  deaths of hosts; and treatment of hosts, leading to the deaths 
of many of the host's parasites.
For births of parasites, there is a transition $X\to X+J$, where~$J$
takes the form 
\eq\label{1-1-transition}
     J_l \Eq 1;\quad J_m \Eq -1; \quad J_j \Eq 0, \ j\neq l,m,
\en
indicating that one $m$-host has become an $l$-host.  For births
of parasites within a host, a transition rate of the form $b_{l-m}mX_m$
could be envisaged, with $l > m$, the interpretation 
being that there are~$X_m$ hosts with parasite burden~$m$, 
each of which gives birth to~$s$ offspring at
rate~$b_s$, for some small values of~$s$.  For infection of an $m$-host,
a possible transition rate would be of the form
\[
      X_m \sum_{j\ge0}N^{-1}X_j \l p_{j,l-m},
\]
since an $m$-host comes into contact with $j$-hosts at a rate proportional
to their density in the host population, and $p_{jr}$ represents the
probability of a $j$-host transferring~$r$ parasites to the infected
host during the contact.  The probability distributions $p_{j\cdot}$
can be expected to be stochastically increasing in~$j$.  Deaths of
parasites also give rise to transitions of the form~\Ref{1-1-transition},
but now with $l < m$, the simplest form of rate being just $dmX_m$ for $l = m-1$,
though~$d=d_m$ could also be chosen to increase with parasite burden.
Treatment of a host would lead to values of $l$ much smaller than~$m$,
and a rate of the form $\k X_m$ for the transition with $l=0$ would
represent fully successful treatment of randomly chosen individuals.
Births and deaths of hosts and immigration all lead to transitions of 
the form
\eq\label{pm1-transition}
     J_l \Eq \pm 1; \quad J_j \Eq 0, \ j\neq l.
\en
For deaths, $J_l=-1$, and a typical rate would be $d'X_l$.  For
births, $J_l = 1$, and a possible rate would be $\sum_{j\ge0}X_j b'_{jl}$ 
(with $l=0$ only, if new-born individuals are free of parasites).
For immigration, constant rates $\l_l$ could be supposed.
Finally, for migration of individual parasites between hosts, transitions are
of the form
\eq\label{migration}
     J_l \Eq J_m \Eq - 1; \quad J_{l+1} \Eq 1;\quad J_{m-1} \Eq 1; \quad
     J_j \Eq 0, \ j\neq l,m,l+1,m-1,
\en
a possible rate being $\g mX_m N^{-1}X_l$.

   For all the above transitions, we can take $J_*=2$ in~\Ref{finite-jumps},
and~\Ref{finite-rates} is satisfied in biologically sensible models.
\Ref{A-assn-1} and~\Ref{A-assn-2} depend on the way in which the matrix~$A$
can be defined, which is more model specific; in practice, \Ref{A-assn-1}
is very simple to check. The choice of~$\m$ in~\Ref{A-assn-2}
is influenced by the need to have~\Ref{F-assn-2} satisfied. 
For Assumptions~\ref{ap-assns}, a possible
choice of~$\n$ is to take $\n(j)=(j+1)$ for each~$j\ge0$, with $S_1(X)$
then representing the number of hosts plus the number of parasites.
Satisfying \Ref{3.5a} is then easy for transitions only involving the
movement of a single parasite, but in general requires assumptions
as to the existence of the $r$-th moments of the distributions
of the numbers of parasites introduced at birth, immigration and infection events.
For~\Ref{U-bnd}, in which transitions involving a net reduction in
the total number of parasites and hosts can be disregarded,
the parasite birth events are those in which the rates typically
have a factor~$mX_m$ for transitions with $J_m=-1$, with~$m$
in principle unbounded.  However,  at such
events, an $m$-individual changes to an $m+s$ individual, with the
number~$s$ of offspring of the parasite being typically small, so that the value
of $J^T\n_r$ associated with this rate has magnitude $m^{r-1}$; the product
$mX_m\,m^{r-1}$, when summed over~$m$, then yields a contribution of
magnitude $S_r(X)$, which is allowable in~\Ref{U-bnd}.  Similar
considerations show that the terms $N^{-1}S_0(X)S_r(X)$ accommodate
the migration rates suggested above.  Finally, in order to have 
Assumptions~\ref{zeta-assns} satisfied,
it is in practice necessary that Assumptions~\ref{ap-assns} are
satisfied for large values of~$r$, thereby imposing restrictions on
the distributions of the numbers of parasites introduced at birth, immigration 
and infection events, as above.

\subsection{Kretzschmar's model}
Kretzschmar~(1993) introduced a model of a parasitic infection, 
in which the transitions from state~$X$ are as follows:
$$
\begin{array}{rlll}
  J\ &\Eq e^{(i-1)} - e\uii &\mbox{at rate}\quad Ni\m x^i,\qquad &i\ge1;\\[1ex]
  J\ &\Eq -e\uii  &\mbox{at rate}\quad N(\k+i\a) x^i,\qquad &i\ge0;\\[1ex]
  J\ &\Eq e^{(0)}  &\mbox{at rate}\quad N\b\sio x^i\theta^i;\\[1ex]
  J\ &\Eq e^{(i+1)} - e\uii &\mbox{at rate}\quad N\l x^i \f(x),
     \qquad &i \ge 0,
\end{array}
$$
where $x := N^{-1}X$, $\f(x) := \|x\|_{11}\{c + \|x\|_1\}^{-1}$ with $c>0$,
and $\|x\|_{11} := \sji j|x|^j$; here, $0 \le \theta \le 1$,
and $\theta^i$ denotes its $i$-th power (our $\theta$ corresponds to the
constant $\xi$ in~\cite{kr93}). Both \Ref{finite-jumps} and~\Ref{finite-rates}
are obviously satisfied.  For Assumptions \Ref{A-assn-1}, \Ref{A-assn-2} 
and~\Ref{F-assn-2}, we note that equation corresponding to~\Ref{F-assn} has
\eqs
   A_{ii} &=& -\{\k + i(\a+\m)\};\quad A^T_{i,i-1} \Eq i\m 
    \ \mbox{and}\ A^T_{i0} \Eq \b\theta^i,\qquad i\ge2;\\
   A_{11} &=& -\{\k + \a+\m\};\quad  A^T_{10} \Eq \m+\b\theta;\\
   A_{00} &=& -\k+\b, \qquad i\ge1,
\ens
with all other elements of the matrix equal to zero, and 
\[
    F^i(x) \Eq \l (x^{i-1} - x^i)\f(x),\quad i\ge1; \qquad  
    F^0(x) \Eq -\l x^0 \f(x).
\]
Hence Assumption \Ref{A-assn-1} is immediate, and Assumption~\Ref{A-assn-2}
holds for $\m(j) = (j+1)^s$, for any $s\ge0$, with $w=(\b-\k)_+$.  For the choice
$\m(j) = j+1$, $F$ maps elements of ${\mathcal R}_{\mu}$ to ${\mathcal R}_{\mu}$, 
and is also locally Lipschitz in the $\m$-norm, with 
$K(\m,F;\Xi) = c^{-2}\l\Xi(2c + \Xi)$. 

For Assumptions~\ref{ap-assns},
choose $\n=\m$; then~\Ref{3.5a} is a finite sum for each $r \ge 0$. Turning to~\Ref{U-bnd},
it is immediate that $U_{0}(x) \le \b S_0(x)$.  Then, for $r\ge1$, 
\eqs
   \sio \l \f(N^{-1}X) X^i\{(i+2)^r - (i+1)^r\} &\le&
          \l \frac{S_1(X)}{S_0(X)}\sio r X^i (i+2)^{r-1} \\
    &\le& r 2^{r-1}\l S_r(X),
\ens
since, by Jensen's inequality, $S_1(X)S_{r-1}(X) \le S_0(X)S_r(X)$.
Hence we can take $k_{r2} = k_{r4} = 0$ and $k_{r1} = \b + r 2^{r-1}\l$
in~ \Ref{U-bnd}, for any $r\ge1$, so that $r_{\max}^{(1)}=\infty$.  
Finally, for~\Ref{V-bnd},
\[
    V_{0}(x) \Le (\k+\b)S_0(x) + \a S_1(x),
\]
so that $k_{03}=\k+\b+\a$ and $k_{05}=0$, and 
\eqs
   V_{r}(x)
 \Le r^2 (\k S_{2r}(x) + \a S_{2r+1}(x)
   +  \m S_{2r-1}(x) +  2^{2(r-1)}\l S_{2r-1}(x)) + \b S_0(x),
\ens
so that we can take $p(r) = 2r+1$, $k_{r3} = \b + r^2\{\k+\a+\m + 2^{2(r-1)}\l\}$,
and $k_{r5}=0$ for any $r\ge1$, and so $r_{\max}^{(2)} = \infty$.  In Assumptions~\ref{zeta-assns},
we can clearly take $r_\m=1$ and $\z(k) = (k+1)^7$, giving $r(\z) = 8$,
$\bbt(\z) = 1$ and $\r(\z,\m) = 17$.

\subsection{Arrigoni's model}
In the metapopulation model of Arrigoni~(2003), the transitions from state~$X$ are
as follows:
$$
\begin{array}{rlll}
  J &=\ e^{(i-1)} - e\uii &\mbox{at rate}\ Ni x^i(d_i + \g(1-\r)),\quad &i\ge2;\\[1ex]
  J &=\ e^{(0)} - e\ui &\mbox{at rate}\ N x^1(d_1 + \g(1-\r) + \k);\\[1ex]
  J &=\ e^{(i+1)} - e\uii &\mbox{at rate}\ Ni b_i x^i,
       \quad &i \ge 1;\\[1ex]
  J &=\ e^{(0)} - e\uii &\mbox{at rate}\ N x^i \k, \quad &i\ge2;\\[1ex]
  J &=\ e^{(k+1)} - e^{(k)} + e^{(i-1)} - e\uii &\mbox{at rate}\ N i x^i x^k \r\g,
     \qquad\qquad k\ge0,\ &i\ge1;
\end{array}
$$
as before, $x := N^{-1}X$.
Here, the total number~$N = \sjo X_j = S_0(X)$ of patches remains constant throughout, 
and the number
of animals in any one patch changes by at most one at each transition; in the
final (migration) transition, however, the numbers in two patches change simultaneously.
In the above transitions, $\gamma, \rho, \kappa$ are non-negative, and $(d_i), (b_i)$ are sequences of non-negative numbers.

Once again, both \Ref{finite-jumps} and~\Ref{finite-rates}
are obviously satisfied.  The equation corresponding to~\Ref{determ} can now be
expressed by taking
\eqs
   A_{ii} &=& -\{\k + i(b_i+d_i+\g)\};\quad A^T_{i,i-1} \Eq i(d_i+\g);
     \quad A^T_{i,i+1} \Eq i b_i, \qquad i\ge1; \\       
   A_{00} &=& -\k, 
\ens
with all other elements of~$A$ equal to zero, and 
\[
    F^i(x) \Eq \r\g \|x\|_{11}(x^{i-1} - x^i),\quad i\ge1; \qquad  
    F^0(x) \Eq -\r\g x^0 \|x\|_{11} + \k,
\]
where we have used the fact that $N^{-1}\sjo X_j = 1$.
Hence Assumption \Ref{A-assn-1} is again immediate, and Assumption~\Ref{A-assn-2}
holds for $\m(j) = 1$ with $w=0$, for $\m(j) = j+1$ with $w = \max_i(b_i-d_i-\g-\k)_+$ (assuming $(b_i)$ and $(d_i)$ to be such that this is finite), 
or indeed for $\m(j) = (j+1)^s$ with any $s\ge2$, with appropriate choice of~$w$.  
With the choice $\m(j) = j+1$, $F$ again maps elements of ${\mathcal R}_{\mu}$ to  
${\mathcal R}_{\mu}$, and is also locally Lipschitz in the $\m$-norm, with 
$K(\m,F;\Xi) = 3\r\g\Xi$.

To check Assumptions~\ref{ap-assns},
take $\n=\m$; once again, \Ref{3.5a} is a finite sum for each $r$. Then, for~\Ref{U-bnd},
it is immediate that $U_{0}(x) = 0$.  For any $r\ge1$, using arguments from
the previous example,
\eqs
   U_{r}(x) &\le& 
      r2^{r-1}\Blb \sii ib_i x^i (i+1)^{r-1} + \sii\sko i\r\g x^ix^k (k+1)^{r-1}\Brb
           \\
    &\le& r 2^{r-1}\{\max_i b_i\, S_r(x) + \r\g S_1(x)S_{r-1}(x)\}\\
    &\le& r 2^{r-1}\{\max_i b_i\, S_r(x) + \r\g S_0(x)S_{r}(x)\},
\ens
so that, since $S_0(x) = 1$, we can take  
$k_{r1} = r 2^{r-1}(\max_i b_i + \r\g)$ and $k_{r2} = k_{r4} = 0$ in~ \Ref{U-bnd}, 
and $r_{\max}^{(1)} = \infty$.  Finally, for~\Ref{V-bnd},
$V_{0}(x) = 0$ and, for $r \ge 1$,
\eqs
   \lefteqn{V_{r}(x)}\\
 &\le& r^2 \Blb 2^{2(r-1)}\max_i b_i\, S_{2r-1}(x) + \max_i(i^{-1}d_i) S_{2r}(x)
    + \g(1-\r)S_{2r-1}(x) \right.\\
  &&\qquad\mbox{}\left. + \r\g (2^{2(r-1)}S_1(x)S_{2r-2}(x) + S_0(x)S_{2r-1}(x))
    \vphantom{\max_i b_i}\Brb  +  \k S_{2r}(x),
\ens
so that we can take $p(r) = 2r$, and (assuming $i^{-1} d_i$ to be finite)
\[
   k_{r3} = \k + r^2\{2^{2(r-1)}(\max_i b_i + \r\g) + \max_i(i^{-1}d_i) + \g\},
\]
and $k_{r5}=0$ for any $r\ge1$, and $r_{\max}^{(2)}= \infty$.  In Assumptions~\ref{zeta-assns},
we can again take $r_\m=1$ and $\z(k) = (k+1)^7$, giving $r(\z) = 8$,
$\bbt(\z) = 1$ and $\r(\z,\m) = 16$.

\section*{Acknowledgement} 
We wish to thank a referee for recommendations that have substantially streamlined 
our arguments.
ADB wishes to thank both the Institute for Mathematical Sciences of the 
National University of Singapore and the Mittag--Leffler Institute 
for providing a welcoming environment while part of this work was
accomplished.  MJL thanks the University of Z\"urich for their hospitality on a number of visits.

\end{document}